\documentclass[11pt]{amsart}

\usepackage{amssymb}
\usepackage{amsmath}
\usepackage{amsfonts}
\usepackage{graphicx}
\usepackage{amsthm}
\usepackage{enumerate}
\usepackage[mathscr]{eucal}
\usepackage{mathrsfs}
\usepackage{verbatim}
\usepackage{xcolor}
\usepackage{hyperref}
\usepackage{wrapfig}
\usepackage{color}

\makeatletter
\@namedef{subjclassname@2010}{%
  \textup{2010} Mathematics Subject Classification}
\makeatother

\numberwithin{equation}{section}

\theoremstyle{plain}
\newtheorem{theorem}{Theorem}[section]
\newtheorem{lemma}[theorem]{Lemma}
\newtheorem{proposition}[theorem]{Proposition}

\newtheorem{defn}[theorem]{Definition}

\theoremstyle{plain}

\numberwithin{equation}{section}

\theoremstyle{remark}

\newcommand\R{\ensuremath{\mathbb{R}}}

\newcommand\Z{\ensuremath{\mathbb{Z}}}
\newcommand\N{\ensuremath{\mathbb{N}}}
\newcommand\eps{\ensuremath{\varepsilon}}

\begin{document}
\date{\today}

\title
{Spectral stability under removal of small segments}

\author{Xiang He}
\thanks{Partially supported by  National Key R and D Program of China 2020YFA0713100, and
	by NSFC no. 12171446, 11721101.}
\address{School of Mathematical Sciences\\
	University of Science and Technology of China\\
	Hefei, 230026\\ P.R. China\\}
\email{hx1224@mail.ustc.edu.cn}

\begin{abstract}
  In the present paper, we deepen the works of L. Abatangelo, V. Felli, L. Hillairet and C. L\'{e}na on the asymptotic estimates of the eigenvalue variation under removal of segments from the domain in $\R^2$. We get a sharp asymptotic estimate when the eigenvalue is simple and the removed segment is tangent to a nodal line of the associated eigenfunction. Moreover, we extend their results to the case when the eigenvalue is not simple.
\end{abstract}


\maketitle

\section{Introduction}

Let $\Omega \subset \R^n$ be a bounded domain. According to classical results in spectral theory, the Dirichlet Laplacian on $\Omega$ admits a sequence of real eigenvalues tending to infinity
\[
0< \lambda_1(\Omega) < \lambda_2(\Omega) \leq \dots \leq \lambda_N(\Omega) \leq \cdots
\]
which encodes lots of geometric information of $\Omega$. A natural question is to study the spectral stability under the removal of a ``small" subset (with Dirichlet boundary condition on the new boundary). This type of problem was first studied by M. Kac in \cite{Kac74} and J. Rauch, M. Taylor in \cite{RT75}. In particular, J. Rauch and M. Taylor showed that if $K$ is a subset with Newtonian capacity zero, then $\lambda_k(\Omega \setminus K)=\lambda_k(\Omega)$ for all $k$. Their result was refined by G. Courtois \cite{GC95} to a quantitative version, namely, for any compact subset $K$ of $\Omega$,
\[
0 \le \lambda_{k}(\Omega \setminus K) -\lambda_k(\Omega)  \le C \ \mathrm{Cap}_\Omega(K),
\]
where $\mathrm{Cap}_\Omega(K)$ is the capacity of $K$ in $\Omega$ (see \S 2 for the definition) and $C$ is a constant that is independent of $K$ (but depends on the eigenspace of $\lambda_k(\Omega)$).

Recently, L. Abatangelo, V. Felli, L. Hillairet and C. L\'{e}na observed in \cite{AFHL18} that in the case where $K=K_\eps \subset \Omega$ ($\eps>0$) is a family of compact sets concentrating to a set $K_0$ of capacity 0 (see \S 2 for the exact meaning), if $\lambda_N(\Omega)$ is a simple eigenvalue, then what really hidden behind G. Courtois's proof is the following sharp asymptotic expansion
\begin{equation}\label{asy exp}
\lambda_N (\Omega \setminus K_\eps) = \lambda_N (\Omega) + \mathrm{Cap}_\Omega (K_\eps,u_N) + \mathrm{o} ( \mathrm{Cap}_\Omega (K_\eps,u_N) ), \text{ as } \eps \to 0^+,
\end{equation}
where $u_N$ is the $L^2$-normalized eigenfunction associated to $\lambda_N (\Omega)$, and $\mathrm{Cap}_\Omega (K_\eps,u_N)$ is the $u_N$-capacity whose definition will be given in \S 2.
They further studied the special case where $\Omega \subset \R^2$ is a region containing the origin $(0,0)$, and the removed compact set $K_\eps$ in $\Omega$ is a small segment $s_\eps$ near $(0,0)$,
\[
s_\eps = [-\eps,\eps] \times \{0\}.
\]

Since $\Omega$ is a planar region, it is well-known that (\cite{FTF10}) if $(0,0)$ is a zero point of $u_N$ of order $k$, then there exists $\beta \in \R \setminus \{0\}$ and $\alpha \in [0,\pi)$ such that
\begin{equation}\label{app at 0}
r^{-k} u_N (r \cos t,r \sin t) \to \beta \sin (\alpha-kt)
\end{equation}
in $C^{1,\tau} ([0,2\pi])$ as $r \to 0^+$ for any $\tau \in (0,1)$.
Under the assumption that $\lambda_N(\Omega)$ is simple, they proved that as $\varepsilon \to 0$ (c.f. Theorem 1.10 in \cite{AFHL18})
\begin{enumerate}
  \item if $u_N(0,0) \ne 0$, then
  \[ \lambda_N (\Omega \setminus s_\eps) - \lambda_N (\Omega)= \frac {2\pi} {|\log\eps|} u_N^2(0) (1+\mathrm{o}(1)).\]
 \item if $(0,0)$ is a zero of $u_N$ of order $k$, and $\alpha \ne 0$ in  (\ref{app at 0}), then
\[\lambda_N (\Omega \setminus s_\eps) - \lambda_N (\Omega)= \eps^{2k} \pi \beta^2 \sin^2 \alpha C_k (1+\mathrm{o}(1))
\]
  where
  \begin{equation} \label{Ck}
  C_k = \sum_{j=1}^k j |A_{j,k}|^2,\ A_{j,k} = \frac {1} {\pi} \int_0^{2\pi} (\cos\eta)^k \cos(j\eta) \mathrm{d}\eta.
  \end{equation}
\item if $(0,0)$ is a zero of $u_N$ of order $k$, and $\alpha=0$ in (\ref{app at 0}), then
\[\lambda_N (\Omega \setminus s_\eps) - \lambda_N (\Omega)=
\mathrm{O}(\eps^{2k+2}).\]
\end{enumerate}
Note that the last case is not completely well-understood even to the leading order: $\alpha=0$ means that the segment $s_\eps$ is tangent to a nodal line of $u_N$, and they expected that in this case ``the vanishing order will depends on the precision of the approximation between the nodal line and the segment".

Furthermore, L. Abatangelo, C. L\'{e}na and P. Musolino \cite{ALM22} studied the problem when $\lambda_N(\Omega)$ is an eigenvalue of multiplicity $m$, in which case after removing $K_\eps$ the eigenvalues will split, and  the asymptotics becomes
\begin{equation}\label{asy expmul}
\lambda_{N+i-1} (\Omega \setminus K_\eps) = \lambda_N (\Omega) + \mu_i^\varepsilon + \mathrm{o} ( \chi_\varepsilon^2), \text{ as } \eps \to 0^+
\end{equation}
for $1 \le i \le m$, where $\mu_i^\varepsilon$ and $\chi_\varepsilon^2$ will be explained  in \S 4 below. Using analytical of eigenfunctions, they also get sharp asymptotics for the case $K_\eps=\eps \overline{\omega}$, where $\omega \subset \mathbb R^2$ is a bounded open domain containing $(0,0)$ such that $\mathbb R^2 \setminus \overline{\omega}$ is connected and $\Omega$ satisfies the same conditions as $\omega$.

The purpose of this short paper is two-fold. Firstly, we will verify their expectation. Precisely, when $\alpha=0$ in (\ref{app at 0}), suppose that the segment $s_\eps$ is tangent to a nodal line of $u_N$. We parameterize the  nodal line of $u_N$ near $(0,0)$ by $(t,f(t))$. We say $s_\eps$ is tangent to this nodal line to the order $l-1$ if there exists $l \in \Z_{\geq 2} \cup \{+\infty\}$ such that
 \begin{equation} \label{van of nodal}
 \begin{aligned}
 &f^{(s)}(0) = 0,\ \forall\ 0 \leq s < l,\ f^{(l)}(0) \neq 0 & \text{ if } l \neq +\infty,\\
 &f^{(s)}(0) = 0,\ \forall\ s \in \Z_{\geq 0} & \text{ if } l = +\infty.
 \end{aligned}
 \end{equation}
Our first main result is

\begin{theorem} \label{tangent simple}
Let $\Omega \subset \R^2$ be a bounded domain containing the origin $(0,0)$. Suppose $\lambda_N(\Omega)$ is a simple Dirichlet eigenvalue of $\Omega$, with $L^2$-normalized eigenfunction $u_N$, and suppose that $(0,0)$ is a zero point of $u_N$ of order $k$ and  the segment $s_\eps=[-\eps,\eps] \times \{0\}$ is tangent to a nodal line of $u_N$ at $(0,0)$ to the order $l-1$.
\begin{enumerate}
  \item If $l = +\infty$, then $\lambda_N (\Omega \setminus s_\eps) = \lambda_N (\Omega)$.
  \item If $l \neq +\infty$, then
  \[
  \lambda_N (\Omega \setminus s_\eps) - \lambda_N (\Omega) = \eps^{2(k+l-1)} [ \binom {k+l-1} {k-1} \beta k! f^{(l)}(0) ]^2 \pi C_{k+l-1} (1+\mathrm{o}(1)),
  \]
  as $\eps \to 0^+$ where $k$ is given in (\ref{app at 0}) and  $C_{k+l-1}$ is given in (\ref{Ck}).
\end{enumerate}
\end{theorem}

Secondly, we will use the techniques in \cite{ALM22} to extend Theorem 1.10 in \cite{AFHL18} to the cases where the eigenvalue $\lambda_N(\Omega)$ is not simple, i.e.,
\[
\lambda=\lambda_N(\Omega)=\cdots=\lambda_{N+m-1}(\Omega)
\]
is a Dirichlet eigenvalue of $\Omega$ with multiplicity $m$. It turns out that the difference of the eigenvalues will depend on the order of \emph{$x_1$-vanishing} (which will be defined in \S 4) of the associated eigenfunctions. Roughly speaking, we will decompose the eigenspace $E(\lambda)$ of $\lambda$ to
\[
E_\infty \oplus E_1 \oplus \cdots \oplus E_p
\]
such that $\dim E_j = 1$ for all $1 \le j \le p$ and the order of \emph{$x_1$-vanishing} of functions in $E_j$ is $k_j$. Suppose that $u_{N+m-p+j-1}$ is an $L^2$-normalized function in $E_j$, then our second main result is

\begin{theorem}\label{main-mul}
Let $\Omega \subset \R^2$ be a bounded domain containing the origin $(0,0)$, and $s_\eps=[-\eps,\eps] \times \{0\}$.
Suppose $\lambda_N(\Omega)$ is a Dirichlet eigenvalue of multiplicity $m$. Then when $\eps$ is small enough, for all $1 \leq i \leq m-p$,
\[
\lambda_{N+i-1} (\Omega \setminus s_\eps) = \lambda_N (\Omega)
\]
and for all $1 \leq j \leq p$,
\begin{equation} \label{eq:asymptEVOrder}
\lambda_{N+m-p+j-1} (\Omega \setminus s_\eps) - \lambda_N (\Omega)
=  \pi C_{k_j} \cdot  [\frac  {\partial^{k_j}_{x_1} u_{N+m-p+j-1}(0,0)} {k_j!}]^2 \rho_{k_j}^\eps + o(\rho_{k_j}^\eps)
\end{equation}
as $\eps \to 0^+$, where $C_{k_j}$ is given in (\ref{Ck}) for $k_j \ge 1$ and $C_0=2$, and %For $k \in \N$ and $\eps > 0$, let
\begin{equation} \label{eq:scale}
\rho_k^\eps =
\begin{cases}
\frac 1 {|\log(\eps)|} & \mbox{ if } k=0\, ,\\
\eps^{2k}		    & \mbox{ if } k\ge1\, .
\end{cases}
\end{equation}
\end{theorem}

\textbf{Acknowledgments.} The author would like to thank his advisor, Zuoqin Wang, for numerous help during various stages of the work.

\section{preparation}

We start with some definitions. For any compact subset $K$ of $\Omega$,
the capacity of $K$ in $\Omega$
 is defined as
\begin{equation} \label{cap}
\mathrm{Cap}_\Omega (K) = \inf \{ \int_\Omega |\nabla f|^2 :\ f-\eta_K \in H^1_0(\Omega\setminus K) \}
\end{equation}
where $\eta_K \in C_c^\infty (\Omega)$ is a fixed function with $\eta_K \equiv 1$ in a neighborhood of $K$. The infimum (\ref{cap}) is achieved by a unique function $V_K \in \eta_K + H^1_0 (\Omega \setminus K)$ so that
\[
\mathrm{Cap}_\Omega (K) = \int_\Omega |\nabla V_K|^2 \mathrm{d}x.
\]
$V_K$ is called the capacitary potential of $K$ in $\Omega$.

More generally, for any  $u \in H^1_0 (\Omega)$, one can define the $u$-capacity of $K$ in $\Omega$ to be
\begin{equation} \label{u-cap}
  \mathrm{Cap}_\Omega (K,u) = \inf \{ \int_\Omega |\nabla f|^2 \mathrm{d}x :\ f-u \in H^1_0 (\Omega \setminus K) \}.
  \end{equation}
  The infimum (\ref{u-cap}) is also achieved by a unique function $V_{K,u} \in u + H^1_0 (\Omega \setminus K)$ such that
  \[
  \mathrm{Cap}_\Omega (K,u) = \int_\Omega |\nabla V_{K,u}|^2 \mathrm{d}x
  \]
  and $V_{K,u}$ is called the potential associated with $u$ and $K$. Note that if $u \in H^1_0 (\Omega \setminus K)$, then $\mathrm{Cap}_\Omega (K,u)=0$.

  The concept of $u$-capacity can be further extended to $H^1_\mathrm{loc} (\Omega)$ functions. For function $v$ in $H^1_\mathrm{loc} (\Omega)$, one can simply define the $v$-capacity to be
  \[
  \mathrm{Cap}_\Omega (K,v) = \mathrm{Cap}_\Omega (K, \eta_K v)
  \]
  with $\eta_K$ as in (\ref{cap}).

Next, we define the meaning of ``concentrating compact sets".
  Let $\{ K_\eps \}_{ \eps>0}$ be a family of compact sets contained in $\Omega$. We say that $K_\eps$ is concentrating to a compact set $K \subset \Omega$ if for every open set $U \subset \Omega$ such that $U \supset K$, there exists a constant  $\eps_U > 0$ such that
  \[U \supset K_\eps, \qquad \forall \eps < \eps_U.\]

Last, we state a proposition that will be used later.

\begin{proposition}[Proposition 2.7 in \cite{AFHL18}] \label{asysegcap}
Let $\Omega \subset \R^2$ be a bounded domain containing the origin $(0,0)$. For function $u \in C^{k+1}_\mathrm{loc} (\Omega) \setminus \{0\}$, suppose the Taylor polynomial at $(0,0)$ of $u$ of order $k$ has the form
\[
P_k (x_1,x_2) = \sum_{j=0}^k c_j x_1^{k-j} x_2^j
\]
for some $c_0,c_1, \cdots,c_k \in \R$, $(c_0,c_1, \cdots,c_k) \neq (0,0, \cdots,0)$.
\begin{enumerate}
  \item If $c_0 \neq 0$, then
  \begin{equation}
  \mathrm{Cap}_\Omega (s_\varepsilon,u)=
  \begin{cases}
  \frac {2\pi} {|\log \eps|} u^2(0) (1+\mathrm{o}(1)), &\text{ if } k=0,\\
  \eps^{2k} c_0^2 \pi C_k (1+\mathrm{o}(1)), &\text{ if } k \geq 1,
  \end{cases}
  \end{equation}
  as $\eps \to 0^+$ and $C_k$ being defined in (\ref{Ck}).
  \item If $c_0 = 0$, then $\mathrm{Cap}_\Omega (s_\eps,u) = \mathrm{O} (\eps^{2k+2})$ as $\eps \to 0^+$.
\end{enumerate}
\end{proposition}

\section{The Proof of Theorem \ref{tangent simple}}

By calculations, one can find that for any $\alpha \in \Z_{\geq 0}$,
\[
\frac {\mathrm{d}^\alpha} {\mathrm{d}t^\alpha} u_N (t,f(t))|_{t=0}
\]
can be written as a finite sum of expressions of the form
\[
\frac {\partial^\beta} {\partial^{\beta-\gamma}_{x_1} \partial^\gamma_{x_2}} u_N(0,0) \times f^{(w_1)}(0) \times \cdots \times f^{(w_\gamma)}(0),
\]
where $\beta, \gamma \in \Z_{\geq 0}$, $w_i \in \N$  ($1 \leq i \leq \gamma$) with
\[
\sum_{j=1}^\gamma w_j = \alpha-\beta+\gamma.
\]
By assumptions (\ref{app at 0}) and (\ref{van of nodal}), one has
\begin{equation} \label{sums}
\frac {\partial^\beta} {\partial^{\beta-\gamma}_{x_1} \partial^\gamma_{x_2}} u_N(0,0) \times f^{(w_1)}(0) \times \cdots \times f^{(w_\gamma)}(0)=0
\end{equation}
when $\beta < k$ or $w_i < l$ for some $1 \leq i \leq \gamma$ and by the fact $(t,f(t))$  parameterize a nodal line of $u_N$, one has
\begin{equation} \label{der-uN-f}
\frac {\mathrm{d}^\alpha} {\mathrm{d}t^\alpha} u_N (t,f(t))|_{t=0} = 0
\end{equation}
for all $\alpha \in \Z_{\geq 0}$. Then combining (\ref{sums}) and (\ref{der-uN-f}), one has
\begin{equation} \label{van-der-x1}
\begin{aligned}
& \frac {\partial^m} {\partial^m_{x_1}} u(0,0) = 0,\ \forall\ 0 \leq m < k+l-1, & \text{ if } l \neq +\infty,\\
& \frac {\partial^m} {\partial^m_{x_1}} u(0,0) = 0,\ \forall\ m \in \Z_{\geq 0}, & \text{ if } l = +\infty.
\end{aligned}
\end{equation}
Thus, when $l = +\infty$, the segment $s_\eps$ belongs to a nodal line of $u_N$. So $\mathrm{Cap}_\Omega (s_\eps,u_N) = 0$ and then, by expansion (\ref{asy exp}), $\lambda_N (\Omega \setminus s_\eps) = \lambda_N (\Omega)$. This proves the first part of Theorem \ref{tangent simple}.

For the case $l\neq+\infty$, let
\begin{equation} \label{taylor}
u_{\#,k+l-2} (x_1,x_2) = \sum_{(h,j) \in \Z_{\geq0}^2 \atop h+j \le k+l-2}\frac{\partial^h_{x_1}\partial^j_{x_2}u(0,0)}{h!j!}x^h_1x^j_2,
\end{equation}
then by (\ref{van-der-x1}) with $l \neq +\infty$, one has $u_{\#,k+l-2}=0$ on the segment $s_\eps$. Thus
\[
\mathrm{Cap}_\Omega (s_\eps,u_N) = \mathrm{Cap}_\Omega (s_\varepsilon, u_N - u_{\#,k+l-2}).
\]
By calculation, one has
\[
0=\frac {\mathrm{d}^{k+l-1}} {\mathrm{d}t^{k+l-1}} u_N (t,f(t))|_{t=0}=\frac {\partial^{k+l-1}} {\partial_{x_1}^{k+l-1}} u(0,0) + \binom {k+l-1} {k-1} \frac {\partial^k} {\partial^{k-1}_{x_1} \partial_{x_2}} u(0,0) f^{(l)}(0).
\]
So
\[
\frac {\partial^{k+l-1}} {\partial_{x_1}^{k+l-1}} u(0,0)= -\binom {k+l-1} {k-1} \frac {\partial^k} {\partial^{k-1}_{x_1} \partial_{x_2}} u(0,0) f^{(l)}(0)= \binom {k+l-1} {k-1} \beta k! f^{(l)}(0)
\]
where the second equal sign is given by (\ref{app at 0}). Then by Proposition \ref{asysegcap}, one has
\[
\mathrm{Cap}_\Omega (s_\varepsilon, u_N - u_{\#,k+l-2}) = \eps^{2(k+l-1)} [ \binom {k+l-1} {k-1} \beta k! f^{(l)}(0) ]^2 \pi C_{k+l-1} ( 1+\mathrm{o}(1) )
\]
as $\eps \to 0^+$. Finally  by expansion (\ref{asy exp}), one can get the vanishing order of $\lambda_N ( \Omega \setminus s_\eps ) - \lambda_N (\Omega)$ when $\alpha=0$. This completes the proof of Theorem \ref{tangent simple}.

\section{multiple case}

Now suppose $\Omega$ is  a bounded domain in $\R^2$ containing $(0,0)$ and
\[
\lambda=\lambda_N(\Omega)=\cdots=\lambda_{N+m-1}(\Omega)
\]
is an eigenvalue of the Dirichlet Laplacian on $\Omega$ with multiplicity $m$, then we will use Proposition 3.1 in \cite{ALM22} to show how the eigenvalues perturb when $\Omega$ is removed by $s_\eps$. In particular, the proof is just a small modification of the proof of Theorem 1.17 in \cite{ALM22}.

\par Firstly, we introduce the concept of $(u,v)$-capacity:

\begin{defn}[Definition 1.7 in \cite{ALM22}]
Given $u,v \in H^1_0 (\Omega)$, the $(u,v)$-capacity of a compact set $K \subset \Omega$ is defined as the following bilinear form
\[
\mathrm{Cap}_\Omega (K,u,v)=\int_\Omega \nabla V_{K,u} \cdot \nabla V_{K,v} \mathrm{d}x,
\]
where $V_{K,u}$ (resp. $V_{K,v}$) is the potential associated with $u$ (resp. $v$) and $K$.
\end{defn}

For a function $u$ on $\Omega$ which is real-analytic in a neighborhood of $0$, we call the smallest $\alpha$ such that
\[
\partial^\alpha_{x_1} u(0,0) \neq 0
\]
the order of \emph{$x_1$-vanishing} of $u$ and we denote it by $\kappa_1 (u)$. In particular, if
\[
\partial^\beta_{x_1} u(0,0)=0, \qquad \forall \beta \in \Z_{\geq 0},
\]
we let $\kappa_1 (u) = +\infty$.

For simplicity,   denote
\begin{center}
$\lambda_{i} (\Omega \setminus s_\eps)$ by $\lambda_i^\eps$, $\qquad V_{s_\eps,u}$ by $V^\eps_u$
\end{center}
and the eigenspace associated to $\lambda$ by $E(\lambda)$. Then directly by Theorem 1.9 in \cite{ALM22}, one has
\begin{equation} \label{not-sharp-est}
\lambda^\eps_{N+i-1}-\lambda = \mu^\eps_i + \mathrm{o}(\chi_\eps^2),\ \forall 1 \leq i \leq m,
\end{equation}
where
\[
\chi_\eps^2 = \sup \{ \mathrm{Cap}_\Omega (s_\eps,u):\ u \in E(\lambda) \text{ and } \|u\|=1 \}
\]
and $\{ \mu_i^\eps \}_{i=1}^m$ are eigenvalues of the quadratic form $r_\eps$ on $E(\lambda)$ defined  by
\begin{equation} \label{r-eps}
r_\eps (u,v) = \mathrm{Cap}_\Omega (s_\eps,u,v) - \lambda \int_\Omega V^\eps_u \cdot V^\eps_v \mathrm{d}x.
\end{equation}

To get more accurate estimate of the difference of   eigenvalues, one need to decompose the eigenspace $E(\lambda)$ with respect to the order of $x_1$-vanishing. Specifically, for any $k \in \Z_{\geq 0}$, consider the mapping $\Pi_k^1: E(\lambda) \to \R_k [x_1]$ that associates to a function its order $k$ Taylor expansion with respect to $x_1$ variable at $(0,0)$, i.e.
\[
\Pi_k^1 u = \sum_{\beta=0}^k \frac 1{\beta !}\partial^\beta_{x_1} u(0,0) x_1^\beta,
\]
and denote the kernel of $\Pi_k^1$ by $N_k$. Let
\[
N_{-1} = E(\lambda),\qquad E_\infty = \underset {k \in \Z_{\ge 0}} \bigcap N_k
\]
and
\[
k_1 > \cdots > k_p \geq 0
\]
be the integers at which there is a jump in the non-increasing sequence $\{ \dim (N_k) \}_{k \geq 1}$. For $1 \leq j \leq p$, let
\[
E_j = N_{k_{j+1}} \cap N_{k_j}^\perp
\]
where we set $k_{p+1} = -1$ and $N_{k_j}^\perp$ is the orthogonal complement of $N_{k_j}$ in $L^2 (\Omega)$. Then, we get an $L^2 (\Omega)$-orthogonal decomposition of $E (\lambda)$,
\[
E(\lambda) = E_\infty \oplus E_1 \oplus \cdots \oplus E_p.
\]
By the construction of the decomposition, one can find that the order of \emph{$x_1$-vanishing} of functions in $E_j$ is $k_j$. In particular, since each function in $\R_k [x_1]$ has only one variable, one has $\dim E_j = 1$ for all $1 \le j \le p$.

Fix an orthonormal basis $\{u_{N+i-1}\}_{i=1}^m$ of $E(\lambda)$ such that
\[
E_\infty = \mathrm{span} \{ u_N, \cdots, u_{N+m-p-1} \},
\]
\[
E_j = \mathrm{span} \{ u_{N+m-p+j-1} \},
\]
for all $1 \leq j \leq p$. Before proving Theorem \ref{main-mul}, we state two lemmas.

\begin{lemma}\label{L2-control}
 For any $f \in H^1_0(\Omega)$,
 \[
  \int_{\Omega} |V_f^\eps|^2 \, dx= o(\mathrm{Cap}_\Omega(s_\eps,f)) \quad \text{as }\eps\to0.
 \]
\end{lemma}

\begin{proof}It's a direct corollary of Lemma A.1 in \cite{AFHL18}.
\end{proof}

\begin{lemma}[Proposition 3.1 in \cite{ALM22}]\label{small-lem}
Let $(\mathcal H, \|\cdot\|)$  be a Hilbert space and $q$ be a quadratic form, semi-bounded from below (not necessarily positive), with
domain $\mathcal D$ dense in $\mathcal H$ and with discrete spectrum $\{ \nu_i \}_{i\geq1}$. Let $\{ g_i \}_{i\geq1}$ be an orthonormal basis of eigenvectors of $q$. Let $N$ and $m$ be positive integers, $F$ be an $m$-dimensional subspace of $\mathcal D$ and $\{ \xi_i^F\}_{i=1}^m$ be the eigenvalues of the restriction of $q$ to $F$.

Assume that there exist positive constants $\gamma$ and $\delta$ such that
\begin{itemize}
 \item[(H1)] $ 0<\delta<\gamma/\sqrt2$;
 \item[(H2)] for all $i\in\{1,\dots,m\}$, $|\nu_{N+i-1}|\le\gamma$, $\nu_{N+m}\ge \gamma$ and, if $N\ge2$, $\nu_{N-1}\le-\gamma$;
 \item[(H3)] $|q(\varphi,g)|\leq \delta\, \|\varphi \|\,\|g\|$ for all $g\in\mathcal D$ and $\varphi \in F$.
\end{itemize}
Then we have
\begin{itemize}
 \item[(i)] $\left|\nu_{N+i-1}- \xi_i^F \right|\le\frac{ 4}{\gamma}\delta^2$ for all $i=1,\ldots,m$;
 \item[(ii)] $\left\| \Pi_N - \mathbb{I}\right\|_{\mathcal L(F,\mathcal H)} \leq { \sqrt 2}\delta/\gamma$,  where $\Pi_N$ is the projection onto the subspace of $\mathcal D$ spanned by $\{g_N,\ldots,g_{N+m-1}\}$.
\end{itemize}
\end{lemma}

\noindent\textbf{Proof of Theorem \ref{main-mul}:}

\par For simplicity, denote
\[
\mu_j = \pi C_{k_j} \cdot  [\frac{\partial^{k_j}_{x_1} u_{N+m-p+j-1}(0,0)} {k_j!}]^2 ,\qquad v^\eps_j = \lambda_{N+m-p+j-1}^\eps - \lambda,
\]
for all $1 \le j \le p$. If $\kappa_1(u)\neq +\infty$, then
\[
\mathrm{Cap}_\Omega (s_\eps,u)= \mathrm{Cap}_\Omega (s_\eps,u-u_{\#,\kappa_1(u)-1})
\]
where $u_{\#,\kappa_1(u)-1}$ is defined in (\ref{taylor}) and if $\kappa_1(u)= +\infty$, then
\[
\mathrm{Cap}_\Omega (s_\eps,u)=0.
\]
So by Proposition \ref{asysegcap}, one has $\chi^2_\eps = \mathrm{O}(\rho^\eps_{k_p})$ as $\eps \to 0^+$. 

Next, we write the matrix of $r_\eps$ (defined in (\ref{r-eps})) under the basis $\{u_{N+i-1}\}_{i=1}^m$. One can relate $E_\infty$ to the kernel of $\mathrm{Cap}_\Omega(s_\eps, \cdot,\cdot)$, and observe a nested structure: $j>i$, $(r_\eps)_{ij}=\mathrm{o}((r_\eps)_{jj})$ and that the diagonal terms are in increasing order. Thus, by the above observations and Lemma \ref{L2-control},  the matrix of $r_\eps$  under the basis $\{u_{N+i-1}\}_{i=1}^m$ is of the form
\begin{equation*}
A_\eps=
\left(
\begin{array}{cccccc}
	&0     					&       &0						   \\
	& 						&\ddots &\vdots						   \\
	&0						& \cdots		&\mu_p \rho^\eps_{k_p}
\end{array}
\right)
+o\left(\rho_{k_p}^\eps\right).
\end{equation*}
So by min-max principle and (\ref{not-sharp-est}), one has
\[
v^\eps_{p} = \mu_p \rho^\eps_{k_p} + \mathrm{o}(\rho^\eps_{k_p}).
\]

For $v^\eps_{p-1}$, we apply Lemma \ref{small-lem} with
\[
\mathcal{H}^\eps = L^2 (\Omega \setminus s_\eps);
\]
\[
q^\eps_{p-1} (u)=\frac {1} {\rho^\eps_{k_p}} ( \int_{\Omega\setminus s_\eps} |\nabla u|^2 \mathrm{d}x - \lambda \int_{\Omega \setminus s_\eps} u^2 \mathrm{d}x ), \text{ for all } u \in H^1_0 (\Omega \setminus s_\eps);
\]
\[
F^\eps_{p-1} = \Pi_\eps (E_\infty \oplus E_1 \oplus \cdots \oplus E_{p-1}),
\]
where $\Pi_\eps u = u- V^\eps_u$. It's easy to check that there exists $\gamma > 0$ such that, for $\eps$ small enough,
\[
|\frac 1 {\rho^\eps_{k_p}} v^\eps_j| \le \gamma, \qquad \text{for } 1 \le j \le m-1;
\]
\[
\frac 1 {\rho^\eps_{k_p}} v^\eps_p \ge 2\gamma;
\]
and in case $N \ge 2$
\[
\lambda^\eps_{N-1} - \lambda \ge -2\gamma.
\]
Similar to the arguments in Section 3.1 in \cite{ALM22}, for all $v \in F^\eps_{p-1}$ and $w \in H^1_0 (\Omega \setminus s_\eps)$, one has
\[
|q^\eps_{p-1} (v,w)| \le \mathrm{o} \bigg( \big( \frac {\rho^\eps_{k_{p-1}}} {\rho^\eps_{k_p}} \big)^{1/2} \bigg) \|v\|_{L^2} \|w\|_{L^2}.
\]
Repeating the proof of Theorem 1.9  in Section 3.2 in \cite{ALM22}, one has
\[
v^\eps_{p-1} = \mu_{p-1} \rho^\eps_{k_{p-1}} + \mathrm{o} (\rho^\eps_{k_{p-1}}).
\]

Continue the above process, one will get
\[
v^\eps_j = \mu_j \rho^\eps_{k_j} + \mathrm{o} (\rho^\eps_{k_j}),\qquad \text{for all } 1 \le j \le p.
\]
Since the proof is vary similar to the proof of Theorem 1.17 in \cite{ALM22}, we omit the details. Last, by the fact that each function in $E_\infty$ is also an eigenfunction of the Dirichlet Laplacian on $\Omega \setminus s_\eps$, one has
\[
\lambda^\eps_{N+i-1} = \lambda,\qquad \text{for all } 1 \le i \le m-p.
\]
This completes the proof of Theorem \ref{main-mul}.
$\hfill\square$

\end{document}